\documentclass[11pt]{amsart}

\usepackage{amsmath}
\usepackage{amssymb}
\usepackage{amscd}
\usepackage{amsthm}
\usepackage{hyperref}
\usepackage{colortbl}

\usepackage{mathtools}
\usepackage{extarrows}

\usepackage[all]{xy}

\usepackage{tikz-cd}

\newtheorem{thm}{Theorem}[section]
\newtheorem*{thm*}{Theorem}
\newtheorem{lemma}[thm]{Lemma}
\theoremstyle{remark}

\newtheorem*{rem}{Remark}

\theoremstyle{definition}
\numberwithin{equation}{section}

\allowdisplaybreaks[4]

\begin{document}

\newcommand{\Q}{{\mathbb Q}}
\newcommand{\Z}{{\mathbb Z}}

\newcommand{\Hom}{\mathop{\mathrm{Hom}}\nolimits}
\newcommand{\Ker}{\mathop{\mathrm{Ker}}\nolimits}
\newcommand{\im}{\mathop{\mathrm{Im}}\nolimits}

\title{Betti numbers of unordered Configuration spaces of the Torus}

\author{Christoph Schiessl}
\address{Departement Mathematik\\ETH Z\"urich\\R\"amistrasse 101\\8092 Z\"urich\\Switzerland}
\email{christoph.schiessl@math.ethz.ch}
\date{\today}

\thanks{The author was supported by the grant ERC-2012-AdG-320368-MCSK. Thanks to Gabriel Drummond-Cole, Javier Fres\'{a}n, Felix Janda, Dan Petersen, Johannes Schmitt, Junliang Shen, Qizheng Yin for very helpful discussions and especially to Rahul Pandharipande for his invaluable support.}
 
\begin{abstract} 
    Using a method of F\'{e}lix and Thomas we compute the Betti numbers of unordered configuration spaces of the torus. 
\end{abstract}

\maketitle

\section{Introduction}
Let $X$ be a topological space. We denote by \[ F_n(X) = \{ (x_1, \dots, x_n) \in X^n \colon x_i \neq x_j \text{ for } i \neq j \}\] the space of \emph{ordered configurations} of $n$ distinct points on $X$. The group ${\mathfrak{S}}_n$ acts freely on $F_n(X)$ by permuting the $n$ points. The quotient \[C_n(X) = F_n(X) / {\mathfrak {S}}_n\] is the space of \emph{unordered configurations} of $n$ points on $X$.

 In the context of representation stability, Church showed that for a connected, orientable manifold $M$ of finite type the rational cohomology groups $H^i(C_n(M), \Q)$ stabilise for $n>i$  \cite[Cor.\! 3]{church}. However, very few of these stable Betti numbers have been explicitly computed. F\'{e}lix and Thomas \cite{felix} showed that for a closed, orientable, nilpotent, even-dimensional manifold $M$, the rational Betti numbers of $C_n(M)$ are determined by the rational cohomology algebra $H^*(M, \Q)$. They constructed an explicit differential graded algebra that we use to compute the Betti numbers of the unordered configuration spaces of the torus $\Sigma_1=S^1 \times S^1$. 

\begin{thm*}
    Suppose $n \ge 2$. Then \[ \dim_{\Q} H^i(C_n(\Sigma_1), \Q) = \begin{cases} 
	\frac{n-2}{2} & i=n+1, \, n \text{ even} \\
	\frac{n+1}{2} & i=n+1, \, n \text { odd} \\
	\frac{3n-4}{2} & i=n, \, n \text{ even} \\
	\frac{3n-1}{2} & i=n, \, n  \text{ odd} \\
	2i-1 & 2 \le i < n \\ 
        2 &  i= 1 \\
        1 & i= 0.  \end{cases}\]
\end{thm*} 

Azam \cite{azam} determined the rational Betti numbers of configuration spaces of Riemann surfaces for $n= 2,3$ in any genus and for $n=4$ in genus~{1} by the Kriz model \cite{kriz}. Napolitano \cite{napolitano} computed the integral cohomology groups of $C_n(\Sigma_1)$ for $n \le 7$ using a cellular decomposition. Drummond-Cole and Knudsen announced computations of Betti numbers for unordered configuration spaces of surfaces for all genera by a different method \cite[p. 30]{knudsen}.
	
The theorem has been tested for all $n \le 20$ using the computer algebra system SAGE \cite{sage}.
\section{Conventions} 
We consider $n \ge 2$ as $C_1(X) \simeq X$. We will always work with cohomology/homology with $\Q$-coefficients and identify \[H^*(M, \Q) = \Hom_{\Q}(H_*(M, \Q), \Q)\]. The free $\Q$-vector space with basis  $x_1$, \dots, $x_n$ is denoted by $\langle x_1, \dots, x_n \rangle$.

For any differential graded commutative algebra $(A,d)$, we use the sign convention $xy = (-1)^{\deg x \deg y} yx$ and $d(xy) = d(x) y + (-1)^{\deg x} x d(y)$ for homogenous $x,y \in A$. We have the free graded commutative algebra $\Lambda(V)$ on any graded vector space $V$ with \[ \Lambda(V) = \text{Exterior algebra }( V^{\text{odd}}) \otimes \text{Symmetric algebra } (V^{\text{even}}).\]
\section{Construction of the algebra}
Let $M$ be a manifold. The cup product gives a map \[\cup \colon H^*(M) \otimes H^*(M) \to H^*(M),\] which dualizes to the diagonal comultiplication \[\Delta \colon H_*(M) \to H_*(M) \otimes H_*(M).\] Using a basis $e_i$ of $H^*(M)$ the map $\Delta$ is given by
\[\Delta(e_k^{*})= \sum_{i,j} (\text{coefficient of } e_k \text{ in } e_i \cup e_j) \  e_i^* \otimes e_j^*,\] where $e_i^*$ denotes the dual basis of $H_*(M)$.

Set $m= \dim(M)$. We take two shifted copies $V, W$ of the vector space $H_*(M)$ with (upper) grading
\begin{align*}
    V^{m-r} = H_r(M) & & W^{2m-1-r} = H_r(M).
\end{align*}
We endow the free graded algebra $\Omega = \Lambda (V \oplus W)$ with the unique differential $D$ of degree 1 such that \begin{align*} D_{|V}=0 & & D_{|W} \colon W \simeq H_*(M) \xrightarrow{\Delta} \Lambda ^2 H_*(M) \simeq \Lambda^2 V. \end{align*}
A lower grading \[ \Omega = \bigoplus_{n \ge 0}  \Omega_n \] can be defined by putting $V$ in degree 1 and $W$ in degree 2. Hence we have \[ \Omega_n = \bigoplus_{r+2s=n} \Lambda^r V \otimes \Lambda^s W .\]
The vectorspace  $ \Omega_n$ is also graded \[ \Omega_n = \bigoplus_{i \ge 0} \Omega^i_n\] by the upper grading inherited from $\Omega$. As $D(W) \subset \Lambda^2 V$, the differential $D$ respects the lower grading and $\Omega_n$ is a subcomplex of $(\Omega, D)$.

F\'{e}lix and Thomas showed that $(\Omega_n, D)$ is a model for the cohomology of $H^*(C_n(M), \Q)$.
\begin{thm} \label{felixthm}
    \cite[Th. A(2)]{felix} Let $M$ be an orientable, closed, nilpotent, even-dimensional manifold. There is an isomorphism of graded vector spaces \[ H^*(C_n(M), \Q) \simeq H^*(\Omega_n, D).\]
\end{thm}
 
\section{Configuration spaces of the torus}

Now we apply this theorem for the torus $\Sigma_1$. Its cohomology algebra is $H^*(\Sigma_1) = \langle 1, a, b, ab \rangle$ with $\deg(a) = \deg(b)=1$ and the relations $ab =- ba$, $a^2 = b^2 = 0$. As $\pi_1(\Sigma_1)=\Z^2$ is abelian and the higher homotopy groups of $\Sigma_1$ vanish, $\Sigma_1$ is a nilpotent space.
We introduce the graded vector spaces $V= \langle v_1, v_a, v_b, v_{ab} \rangle$ and $W= \langle w_1, w_a, w_b, w_{ab} \rangle$ with degrees \begin{align*} \deg v_1 &= 2 & \deg w_1 &= 3 \\ \deg v_a &= 1 & \deg w_a &= 2 \\ \deg v_b &= 1 & \deg w_b &= 2 \\ \deg v_{ab} &= 0 & \deg w_{ab} &= 1. \end{align*}
We look at the graded algebra $\Omega = \Lambda \langle v_1, v_a, v_b, v_{ab}, w_1, w_a, w_b, w_{ab} \rangle$ with the differential $D$ given by 
\begin{align*} 
    D(v_1) &= 0 & D(w_1) & = v_1^2 \\
    D(v_a) &= 0 & D(w_a) & = 2v_1v_a \\
    D(v_b) &= 0 & D(w_b) &=  2v_1v_b \\
    D(v_{ab}) &= 0 & D(w_{ab}) &= 2v_1 v_{ab}+ 2 v_{a}v_b.
\end{align*}
By Theorem \ref{felixthm} we have to compute the cohomology groups of the subcomplexes \[ \Omega_n = \bigoplus_{r+2s= n} \Lambda^r V \oplus \Lambda^s W.\] We will do this by embedding them into the algebra \[\Theta = \Lambda \langle v_1, v_a, v_b, w_1, w_a, w_b, w_{ab} \rangle,\] with differential $d$ given by:
\begin{align*}
    d(v_1) &= 0 & d(w_1) & = v_1^2 \\
    d(v_a) &= 0 & d(w_a) & = 2v_1v_a \\
    d(v_b) &= 0 & d(w_b) &=  2v_1v_b \\
    & & d(w_{ab}) &= 2v_1 + 2 v_{a}v_b.
\end{align*} All variables have the same grading as in $\Omega$; we only set $v_{ab}=1$. 
\begin{lemma} There is an isomorphism $H^i(\Omega_n, D) \simeq H^i(\Theta, d)$ for $i < n$.
\label{first}
\end{lemma}
\begin{proof}
The injective map \[\pi \colon \Omega_n \to \Theta, \ \ v_{ab} \mapsto 1\] respects the grading as $\deg v_{ab} =0$ and commutes with the differentials. Take a degree $i < n$ and any monomial \[\prod v_k^{e_k} \prod w_l^{f_l} \in \Theta^i\] of degree $i$. The only generators of $\Omega$ where the lower degree exceeds the upper one are $v_{ab}$ and $w_{ab}$. As $w_{ab}^2=0$ we see \[ \sum e_k + 2 \sum f_l \le i+1.\] So the monomial \[v_{ab}^{n- \sum e_k -2 \sum f_l} \prod v_k^{e_k} \prod w_l^{f_l}\] is in $\Omega_n$ and \[ \pi ( v_{ab}^{n- \sum e_k -2 \sum f_l} \prod v_k^{e_k} \prod w_l^{f_l} )= \prod v_k^{e_k} \prod w_l^{f_l}.\] Thus $\pi$ is also surjective in degree $i$.
Altogether, $\pi$ induces an isomorphism $H^i(\Omega_n, d) \simeq H^i(\Theta, d)$ for $i< n$.
\end{proof}

In order to compute the Betti numbers of $(\Theta, d)$ we compare $d$ with the simpler differential $d_0$ given by
\begin{align*}
    d_0(v_1) &= 0 & d_0(w_1) & = 0 \\
    d_0(v_a) &= 0 & d_0(w_a) & = 0 \\
    d_0(v_b) &= 0 & d_0(w_b) &=  0 \\
    & & d_0(w_{ab}) &= 2v_1 + 2 v_{a}v_b
\end{align*} 
   
\begin{lemma} There is an isomorphism $\varphi \colon (\Theta, d_0) \to (\Theta, d)$.
\label{def}
\end{lemma}
\begin{proof}
It can be explicitly given by
\begin{align*}
    \varphi(v_1) &= v_1 & \varphi(w_1) &= w_1 - \frac{1}{2} v_1 w_{ab}+\frac{1}{2} v_bw_a \\
    \varphi(v_a) &= v_a & \varphi(w_a) &= w_a + v_a w_{ab} \\
    \varphi(v_b) &= v_b & \varphi(w_b) &= w_b + v_b w_{ab} \\
    & &  \varphi(w_{ab}) &= w_{ab} 
\end{align*}
As $d(\varphi(w_1)) = d(\varphi(w_a)) = d(\varphi(w_b))=0$, the map $\varphi$ commutes with the differentials.
\end{proof}
\begin{lemma}
    The Betti numbers of $H^*(\Theta, d_0)$ are
    \[ \dim_{\Q} H^i(\Theta, d_0) = \begin{cases} 1 & i=0 \\ 2 &  i= 1 \\ 2i-1 & i \ge 2. \end{cases} \]
\label{betti}    
\end{lemma}
    \begin{proof}
        Denote $T = \Lambda \langle v_1, v_a, v_b, w_1, w_a, w_b \rangle$. Then we have $\Theta = T \oplus T w_{ab}$. Observe that $d_0|T =0$. Take homogenous $x,y \in T$. We compute \[ d_0(x+yw_{ab}) = d_0(x) + d_0(y) w_{ab} \pm y d_0(w_{ab}) = \pm 2 y (v_1+v_av_b).\]  As  $v_1$ has even degree, $v_1 + v_av_b$ is not a zero-divisor. So we know that $\Ker(d_0) = T$ and \[H^*(\Theta, d_0)=T/(v_1+v_av_b) \simeq T/(v_1) \simeq \Lambda \langle v_a, v_b, w_1, w_a, w_b \rangle .\]
        The Poincar\'{e} series of $\Lambda \langle v_a, v_b, w_1, w_a, w_b \rangle$ is \[ \frac{(1+t^{\deg v_a})(1+t^{\deg v_b})(1+t^{\deg w_1})}{(1-t^{\deg w_a})(1-t^{\deg w_b})}  = \frac{(1+t)^2(1+t^3)}{(1-t^2)^2}   =  \frac{1+t^3}{(1-t)^2}, \] which expands to  \[1+2t+3t^2+5t^3+7t^4+\dots+(2i-1)t^i+ \cdots . \qedhere \]
\end{proof}
Combining Lemmas \ref{first}, \ref{def} and \ref{betti} we have computed $\dim_\mathbb{Q} H^i(\Omega_n)$ for $i<n$.
\begin{rem}
We consider the morphism 
\[ p \colon \Omega_n \to \Lambda \langle v_a, v_b, w_1, w_a, w_b, w_{ab} \rangle, \, v_{ab} \mapsto 1, \, v_1 \mapsto - v_a v_b .\] The above proof shows that for any $x \in \im D$ necessarily $p(x)=0$.   
\end{rem}

\begin{lemma}
We have \begin{align*} \dim_{\mathbb{Q}} H^{n+1}(\Omega_n) = \begin{cases} \frac{n-2}{2} & n \text{ even} \\ \frac{n+1}{2} & n \text{ odd} \end{cases} & & \dim_{\mathbb{Q}} H^{i}(\Omega_n)=0 \text{ for } i>n+1. \end{align*}
\end{lemma}

\begin{proof} We denote $\Theta' = \Lambda \langle v_1, v_a, v_b, v_{ab}, w_a, w_b, w_{ab} \rangle$.
The only generators with upper grading exceeding the lower grading are $v_1$ and $w_1$. Hence any $x \in \Omega_n^i$ with $i>n$ can be written as $ x= v_1 f + w_1 g$ where $f, g \in \Theta'$. We compute \[ D(x) =v_1 D(f) + v_1^2 g - w_1 D(g).\] As $D(\Theta') \subset \Theta'$ we see that $D(x) =0$ implies $D(g)=0$. So $x \in \Ker D$ if and only if $D(f) = - v_1 g$. Therefore any $x \in \Ker D$ is of the form \[x(f) = v_1 f  - w_1 \frac{D(f)}{v_1}\] with $f \in \Theta'$ such that $v_1 | D(f)$.

We will now discuss when  the cycles $x(f)$ are a boundary. 
If $f = v_1h$ then \[ D(w_1h) = v_1^2h - w_1 D(h) = v_1 f - w_1 \frac{D(f)}{v_1}= x(f).\] For any $x(f) \in \Omega_n^i$ with $i> n+1$ we know that $f$ has to be divisible by $v_1$. Hence $H^i(\Omega_n)=0$ for $i> n+1$.

Now we look at the case $i=n+1$. If $v_{ab} | f$ or $w_{ab} | f$ then $v_1| f$ for degree reasons. 

We consider the sets \[ B_{\text{odd}} =  \{ w_a^{n_1} w_b^{n_2} \mid  2n_1+2n_2+1= n; \, n_1, n_2 \ge 0 \} \] for odd degree $n$ and 
\[ B_{\text{even}} = \{ v_b w_a w_a^{n_1} w_b^{n_2} \mid 2n_1 +2n_2 +4=n; \, n_1, n_2 \ge 0  \} \] for even $n$.

If $v_1 f = D(h)$ then $D(f) = 0$ and hence we get $x(f) = D(h) \in \im(D)$.
Using the relations 
\begin{align*} 
D( v_b w_a^{n_1+1} w_b^{n_2}) & = - 2 (n_1+1) v_1 v_a v_b w_a^{n_1} w_b^{n_2} \\ 
D(w_a^{n_1+1}) &= 2(n_1+1) v_1 v_a w_a^{n_1} \\  D(w_b^{n_1+1}) & =2 (n_1+1) v_1 v_b w_b^{n_1} \\  D( w_a^{n_1+1} w_b^{n_2+1} ) & = 2(n_1+1) v_1 v_a w_a^{n_1} w_b^{n_2+1} + 2 (n_2+1) v_1 v_b w_a^{n_1+1} w_b^{n_2} \end{align*} we conclude that the set $\{ x(b) | b \in B_{\text{even}} \}$ resp. $\{ x(b) | b \in B_{\text{odd}}\}$ is a generating system of $H^{n+1}(\Omega_n)$ for even resp. odd $n$.

By applying $p$ we see that no non-trivial linear combinations of these generating sets are boundaries.
Hence we found an explicit basis of $H^{n+1}(\Omega_n)$.
\end{proof}

\begin{lemma}
We have
\[ \dim_{\mathbb{Q}} H^{n}(\Omega_n) =  \begin{cases} \frac{3n-4}{2} & n \text{ even } \\ \frac{3n-1}{2} &  n \text{ odd} \end{cases}. \]
\end{lemma}

\begin{proof}
As the torus acts freely on $C_n(\Sigma_1)$ we have  $\chi(\Omega_n) =0$. Using the above computation of $\dim_{\mathbb{Q}} H^{n+1}(\Omega_n)$ and 
\[ \sum_{i=0}^{n-1} \dim_{\mathbb{Q}} H^i(\Omega_n) = 1-2+3+ \dots + (-1)^{n-1} (2n-3) = (-1)^{n-1} (n-1) \] we can reconstruct the only missing Betti number $\dim_{\mathbb{Q}} H^{n}(\Omega_n)$.
\end{proof}

Combining all lemmas, we have computed $\dim_{\Q} H^i(C_n(\Sigma_1), \Q)$ for all $i$. We reproduce exactly the stability result \[\dim_{\Q} H^i(C_{n+1}(\Sigma_1), \Q) = \dim_{\Q} H^i(C_n(\Sigma_1), \Q)\] for $n >i$ of Church \cite[Cor. 3]{church}.

\begin{rem} Let $d \ge 1$. 
    With the same method one immediately finds for $n \ge 3$ 
    \[ \dim_{\Q} H^i(C_n(S^{2d}), \Q)= \begin{cases} 1 & \text{ for } i=0, \, 4d-1 \\ 0 & \text{ otherwise}, \end{cases} \] which has also been computed by \cite{randal}, \cite{salvatore}.  
\end{rem}
\begin{rem}
    It seems that our method does not work for surfaces of genus $g>1$ because the differential can not be deformed as in Lemma \ref{def}.
\end{rem}

\bibliographystyle{alpha}
\bibliography{torus.bib}
  
\end{document}